\documentclass[10pt,reqno]{amsart}
\usepackage[foot]{amsaddr}
\usepackage{etoolbox}

\usepackage[english]{babel}
\usepackage[utf8x]{inputenc}
\usepackage[T1]{fontenc}
\usepackage[a4paper,margin=3cm]{geometry}

\usepackage{algorithm}
\usepackage{algorithmic}
\usepackage{subfig}
\usepackage{amsmath}
\usepackage{amssymb}
\usepackage{dsfont}
\usepackage[round]{natbib}

\usepackage{amsthm}
\newtheorem{theorem}{Theorem}
\newtheorem{lemma}{Lemma}

\newtheorem{assumption}{Assumption}
\newtheorem{corollary}{Corollary}

\usepackage{graphicx}
\usepackage[colorinlistoftodos]{todonotes}
\usepackage[colorlinks=true, allcolors=blue]{hyperref}

\patchcmd{\section}{\scshape}{\bfseries\scshape}{}{}
\makeatletter
\renewcommand{\@secnumfont}{\bfseries}
\makeatother
\makeatletter
\def\subsection{\@startsection{subsection}{3}%
  \z@{.5\linespacing\@plus.7\linespacing}{.2\linespacing}%
  {\normalfont\bfseries}}
\makeatother

\newcommand{\ssda}{$SSDA$}
\newcommand{\esdacd}{$ESDACD$}
\newcommand{\msda}{$MSDA$}
\newcommand{\xbar}{\bar{X}}

\title[Accelerated Decentralized Optimization with Local Updates]{Accelerated Decentralized Optimization with Local Updates for Smooth and Strongly Convex Objectives}
\date{}

\author{
Hadrien Hendrikx $^\dagger$ 
\qquad Francis Bach \qquad Laurent Massouli\'e $^\dagger$
}

\address{INRIA - Département d’informatique de l’ENS \\
Ecole normale supérieure, CNRS, INRIA \\ 
PSL Research University, 75005 Paris, France}

\address{$\dagger$ MSR-INRIA Joint Centre}

\email{hadrien.hendrikx@inria.fr}
\email{laurent.massoulie@inria.fr}
\email{francis.bach@inria.fr}

\begin{document}

\maketitle

\begin{abstract}
In this paper, we study the problem of minimizing a sum of smooth and strongly convex functions split over the nodes of a network in a decentralized fashion. We propose the algorithm \esdacd, a decentralized accelerated algorithm that only requires local synchrony. Its rate depends on the condition number $\kappa$ of the local functions  as well as the network topology and delays. Under mild assumptions on the topology of the graph, \esdacd~takes a time $O((\tau_{\max} + \Delta_{\max})\sqrt{{\kappa}/{\gamma}}\ln(\epsilon^{-1}))$ to reach a precision $\epsilon$ where $\gamma$ is the spectral gap of the graph, $\tau_{\max}$ the maximum communication delay and $\Delta_{\max}$ the maximum computation time. Therefore, it matches the rate of \ssda~\citep{scaman2017optimal}, which is optimal when $\tau_{\max} = \Omega\left(\Delta_{\max}\right)$. Applying \esdacd~to quadratic local functions leads to an accelerated randomized gossip algorithm of rate $O(  \sqrt{\theta_{\rm gossip}/n})$ where $\theta_{\rm gossip}$ is the rate of the standard randomized gossip~\citep{boyd2006randomized}. To the best of our knowledge, it is the first asynchronous gossip algorithm with a provably improved rate of convergence of the second moment of the error. We illustrate these results with experiments in idealized settings. 
\end{abstract}

\section{Introduction}
Many modern machine learning applications require to process more data than one computer can handle, thus forcing to distribute work among computers linked by a network. In the typical machine learning setup, the function to optimize can be represented as a sum of local functions $f(x) = \sum_{i=1}^n f_i(x)$, 
where each~$f_i$ represents the objective over the data stored at node $i$. This problem is usually solved incrementally by alternating rounds of gradient computations and rounds of communications~\citep{nedic2009distributed, boyd2011distributed, duchi2012dual, shi2015extra,mokhtari2016dsa, scaman2017optimal, nedic2017achieving}.

Most approaches assume a centralized network with a master-slave architecture in which workers compute gradients and send it back to a master node that aggregates them. There are two main different flavors of algorithms in this case, whether the algorithm is based on stochastic gradient descent \citep{zinkevich2010parallelized, recht2011hogwild} or randomized coordinate descent \citep{nesterov2012efficiency, liu2015asynchronous,liu2015asynchronousjmlr,fercoq2015accelerated, hannah2018texttt}. Although this approach usually works best for small networks, the central node represents a bottleneck both in terms of communications and computations. Besides, such architectures are not very robust since the failure of the master node makes the whole system fail. In this work, we focus on decentralized architectures in which nodes only perform local computations and communications. These algorithms are generally more scalable and more robust than their centralized counterparts~\citep{lian2017can}. This setting can be used to handle a wide variety of tasks~\citep{colin2016gossip}, but it has been particularly studied for stochastic gradient descent, with the D-PSGD algorithm \citep{nedic2009distributed, ram2009asynchronous,ram2010distributed} and its extensions \citep{lian2017asynchronous,tang2018d}.

A popular way to make first order optimization faster is to use Nesterov acceleration~\citep{nesterov2013introductory}. Accelerated gradient descent in a dual formulation yields optimal synchronous algorithms in the decentralized setting \citep{scaman2017optimal, ghadimi2013multi}. Variants of accelerated gradient descent include the acceleration of the coordinate descent algorithm~\citep{nesterov2012efficiency,allen2016even,nesterov2017efficiency}, that we use in this paper to solve the problem in \citet{scaman2017optimal}. This approach yields different algorithms in which updates only involve two neighboring nodes instead of the full graph. Our algorithm can be interpreted as an accelerated version of \citet{gower2015stochastic, necoara2017random}. Updates consist in gossiping gradients along edges that are sequentially picked from the same distribution independently from each other.

Using coordinate descent methods on the dual allows to have local gradient updates. Yet, the algorithm also needs to perform a global contraction step involving all nodes. In this paper, we introduce Edge Synchronous Dual Accelerated Coordinate Descent (\esdacd), an algorithm that takes advantage of the acceleration speedup in a decentralized setting while requiring only \emph{local} synchrony. This weak form of synchrony consists in assuming that a given node can only perform one update at a time, and that for a given node, updates have to be performed in the order they are sampled. It is called the \emph{randomized} or \emph{asynchronous} setting in the gossip literature~\citep{boyd2006randomized}, as opposed to the synchronous setting in which all nodes perform one update at each iteration. Following this convention, we may call \esdacd~an \emph{asynchronous} algorithm. The locality of the algorithm allows parameters to be fine-tuned for each edge, thus giving it a lot of flexibility to handle settings in which the nodes have very different characteristics.

Synchronous algorithms force all nodes to be updated the same number of times, which can be a real problem when some nodes, often called \emph{stragglers} are much slower than the rest. Yet, we show that we match (up to a constant factor) the speed rates of optimal synchronous algorithms such as \ssda~ \citep{scaman2017optimal} even in idealized homogeneous settings in which nodes never wait when performing synchronous algorithms. In terms of efficiency, we match the oracle complexity of \ssda~with lower communication cost. This extends a result that is well-known in the case of averaging, \emph{i.e.}, that randomized gossip algorithms match the rate of synchronous ones~\citep{boyd2006randomized}. We also exhibit a clear experimental speedup when the distributions of nodes computing power and local smoothnesses have a high variance.

Choosing quadratic $f_i$ functions leads to solving the distributed average consensus problem, in which each node has a variable $c_i$ and for which the goal is to find the mean of all variables $\bar{c} = \frac{1}{n}\sum_{i=1}^n c_i$. It is a historical problem \citep{degroot1974reaching, chatterjee1977towards} that still attracts attention \citep{cao2006accelerated, boyd2006randomized, loizou2018accelerated} with many applications for averaging measurements in sensor networks \citep{xiao2005scheme} or load balancing \citep{diekmann1999efficient}. Fast synchronous algorithms to solve this problem exist~\citep{oreshkin2010optimization} but no asynchronous algorithms match their rates. We show that \esdacd~is faster at solving distributed average consensus than standard asynchronous approaches \citep{boyd2006randomized, cao2006accelerated} as well as more recent ones~\citep{loizou2018accelerated} that do not show improved convergence rates for the second moment of the error. The complexity of gossip algorithms generally depends on the smallest non-zero eigenvalue of the gossip matrix $W$, a symmetric semi-definite positive matrix of size $n \times n$ ruling how nodes aggregate the values of their neighbors such that ${\rm Ker}(W) = {\rm Vec}(\mathds{1})$ where $\mathds{1}$ is the constant vector. We improve the rate from $\lambda^+_{\min}(W)$ to $O\left(\frac{1}{\sqrt{n}}\sqrt{\lambda^+_{\min}(W)}\right)$ where $\lambda^+_{\min}(W) \leq \frac{1}{n - 1}$ is the smallest non-zero eigenvalue of the gossip matrix, thus gaining several orders of magnitude in terms of scaling for sparse graphs. In particular, in well-studied graphs such as the grid, we match (up to logarithmic factors that we do not consider) the $O(n^{3/2})$ iterations complexity of advanced gossip algorithms presented by \citet{dimakis2010gossip}.

\section{Model}
The communication network is represented by a graph $\mathcal{G} = (V, E)$. When clear from the context, $E$ will also be used to designate the number of edges. Each node~$i$ has a local function $f_i$ on $\mathbb{R}^d$ and a local parameter $x_i \in \mathbb{R}^d$. The global cost function is the sum of the functions at all nodes: $F(x) = \sum_{i=1}^n f_i(x_i)$ Each $f_i$ is assumed to be $L_i$-smooth and $\sigma_i$-strongly convex, which means that for all $x, y \in \mathbb{R}^{d}$:\vspace{-5pt}
\begin{equation}
f_i(x) - f_i(y) \leq \nabla f_i(y)^T(x - y) + \frac{L_i}{2} \|x - y\|^2
\end{equation}
\begin{equation}
f_i(x) - f_i(y) \geq \nabla f_i(y)^T(x - y) + \frac{\sigma_i}{2} \|x - y\|^2.
\end{equation}
Note that the fenchel conjugate $f_i^*$ of $f_i$ (defined in Equation~\eqref{eq:fenchel}) is ($L_i^{-1}$)-strongly convex and ($\sigma_i^{-1}$)-smooth, as shown in~\citet{kakade2009duality}. We denote $L_{\max} = \max_i L_i$ and $\sigma_{\min} = \min_i \sigma_i$. Then, we denote $\kappa_l = \frac{L_{\max}}{\sigma_{\min}}$. $\kappa_l$ is an upper bound of the condition number of all $f_i$ as well as an upper bound of the global condition number. Adding the constraint that all nodes should eventually agree on the final solution, so the optimization problem can be cast as:
\begin{equation}
\label{eq:primal_problem}
\min_{x \in \mathbb{R}^{n \times d}:\ x_i = x_j \ \forall i,j \in \{1, ..., n \} } F(x).
\end{equation}
We assume that a communication between nodes $i,j \in V$ takes a time $\tau_{ij}$. If $(i,j) \notin E$, the communication is impossible so $\tau_{ij} = \infty$. Node $i$ takes time $\Delta_i$ to compute its local gradient.

\section{Algorithm}
In this section, we specify the Edge Synchronous Decentralized Accelerated Coordinate Descent (\esdacd) algorithm. We first give a formal version in Algorithm~\ref{algo:adacd} and prove its convergence rate. Then, we present the modifications needed to obtain the implementable version given by Algorithm~\ref{algo:adacd_eff}. 

\subsection{Problem derivation}
In order to obtain the algorithm, we consider a matrix $A \in \mathbb{R}^{n \times E}$ such that ${\rm Ker}(A^T) = {\rm Vec}(\mathds{1})$ where $\mathds{1} = \sum_{i=1}^n e_i$ and $e_i \in \mathbb{R}^{n \times 1}$ is the unit vector of size $n$ representing node $i$. Similarly, we will denote $e_{ij} \in \mathbb{R}^{E \times 1}$ the unit vector of size $E$ representing coordinate $(i,j)$. Then, the constraint in Equation~\eqref{eq:primal_problem} can be expressed as $A^T x = 0$ because if $x \in {\rm Ker}(A^T)$ then all its coordinates are equal and the problem writes:
\begin{equation}
\min_{x \in \mathbb{R}^{n \times d}: \ A^T x = 0} F(x).
\end{equation}
This problem is equivalent to the following one:
\begin{equation}
    \min_{x \in \mathbb{R}^{n\times d}} \max_{\lambda \in \mathbb{R}^{E \times d}}  F(x) - \langle \lambda, A^Tx \rangle,
\end{equation}
where the scalar product is the usual scalar product over matrices $\langle x,y \rangle = {\rm Tr}\left(x^T y\right)$ because the value of the solution is infinite whenever the constraint is not met. This problem can be rewritten: 
\begin{equation}
     \max_{\lambda \in \mathbb{R}^{E\times d}} \min_{x \in \mathbb{R}^{n\times d}} F(x) - \langle A \lambda, x \rangle
\end{equation}
because $F$ is convex and $A^T \mathds{1} = 0$. Then, we obtain the dual formulation of this problem, which writes:\vspace{-3pt}
\begin{equation}
\label{eq::dual_formulation}
\max_{\lambda \in \mathbb{R}^{E \times d}} -F^*(A\lambda),\vspace{-3pt}
\end{equation}
where $F^*$ is the Fenchel conjugate of $F$ which is obtained by the following formula:\vspace{-3pt}
\begin{equation}
    \label{eq:fenchel}
    F^*(y) = \max_{x \in \mathbb{R}^{n\times d}} \langle x,y \rangle - F(x).\vspace{-3pt}
\end{equation}
$F^*$ is well-defined and finite for all $y \in \mathbb{R}^{E\times d}$ because $F$ is strongly convex. We solve this problem by applying a coordinate descent method. If we denote $F_A^*: \lambda \rightarrow F^*(A\lambda)$ then the gradient of $F^*_A$ in the direction $(i,j)$ is equal to $\nabla_{ij}F^*_A = e_{ij}^T A^T \nabla F^*$. Therefore, the sparsity pattern of $A e_{ij}$ will determine how many nodes are involved in a single update. Since we would like to have local updates that only involve the nodes at the end of a single edge, we choose $A$ such that, for any $\mu_{ij} \in \mathbb{R}$:
\begin{equation}
\label{eq:A_matrix}
    A e_{ij} = \mu_{ij} (e_i - e_j).
\end{equation}
This choice of $A$ satisfies $e_{ij}^TA^T \mathds{1} = 0$ for all $(i,j) \in E$ and ${\rm Ker}(A^T) \subset {\rm Vec}(\mathds{1})$ as long as $(V, E_+)$ is connex where $E_+ = \{(i,j) \in E, \mu_{ij}>0\}$. Such $A$ happens to be canonical since it is a square root of the Laplacian matrix if all $\mu_{ij}$ are chosen to be equal to $1$. When not explicitly stated, all $\mu_{ij}$ are assumed to be constant so that $A$ only reflects the graph topology. Other choices of $A$ involving more than two nodes per row are possible and would change the trade-off between the communication cost and computation cost but they are beyond the scope of this paper.

\subsection{Formal algorithm}

The algorithm can then be obtained by applying ACDM \citep{nesterov2017efficiency} on the dual formulation. We need to define several quantities. Namely, we denote $p_{ij} \in \mathbb{R}$ the probability of selecting edge $(i,j)$ and $\sigma_A \in \mathbb{R}$ the strong convexity of $F_A^*$. $A^+ \in \mathbb{R}^{E \times n}$ is the pseudo-inverse of $A$ and $\|x\|^2_{A^+A} = x^T A^+A x$ for $x\in \mathbb{R}^{E\times 1}$. Variable $S \in \mathbb{R}$ is such that for all $(i,j) \in {E}$, 
$$e_{ij}^T A^+A e_{ij} \mu_{ij}^2 p_{ij}^{-2}(\sigma_i^{-1} + \sigma_j^{-1}) \leq S^2.$$
We define $\delta = \theta \frac{1 - \theta}{1 + \theta} \in \mathbb{R}$ with\vspace{-3pt}
\begin{equation}
\label{eq:rate_theta}
\theta^2 = \min_{ij} \frac{p_{ij}^2}{\mu_{ij}^2 e_{ij}^T A^+A e_{ij}} \frac{\sigma_A}{\sigma_i^{-1} + \sigma_j^{-1}} \geq \frac{\sigma_A}{S^2}.
\end{equation}
Finally, $\eta_{ij} = \frac{1}{1 + \theta} \big(\mu_{ij}^{-2}(\sigma_i^{-1} + \sigma_j^{-1})^{-1} +(p_{ij}S^2)^{-1}\big) \in \mathbb{R}$ and
\begin{equation}
    g_{ij}(y_t) = e_{ij} e_{ij}^T A^T \nabla F^*(A y_t) \in \mathbb{R}^{E \times d}.
\end{equation}

\begin{algorithm}
\caption{Asynchronous Decentralized Accelerated Coordinate Descent}
\label{algo:adacd}
\begin{algorithmic}
\STATE $y_0 = 0$, $v_0 = 0$, $t = 0$
\WHILE{$t < T$}
\STATE Sample $(i,j)$ with probability $p_{ij}$
\STATE $y_{t+1} = (1 - \delta) y_t + \delta v_t - \eta_{ij} g_{ij}(y_t)$
\STATE $v_{t+1} = (1 - \theta) v_t + \theta y_t - \frac{\theta}{\sigma_A p_{ij}} g_{ij}(y_t)$
\ENDWHILE
\end{algorithmic}
\end{algorithm}

\begin{theorem}
\label{thm:adacd}
Let $y_t$ and $v_t$ be the sequences generated by Algorithm~\ref{algo:adacd}. Then:
\vspace{-5pt}\begin{equation}
 2\left(\mathbb{E}[F_A^*(x_t)] - F_A^*(x^*)\right) + \sigma_A \mathbb{E}[r_t^2] \leq C(1 - \theta)^{t},\vspace{-5pt}
\end{equation}
with $x_t = (1 + \theta) y_t - \theta v_t$, $x^* \in \arg\min_{x} F_A^*(x)$, $r_t^2 = \|v_t - x^*\|^2_{A^+A}$ and $C = r_0^2 + 2\left(F_A^*(x_0) - F_A^*(x^*)\right)$.
\end{theorem}

Theorem~\ref{thm:adacd} shows that Algorithm~\ref{algo:adacd} converges with rate~$\theta$. Lemma~\ref{lemma:sc_FA}, in Appendix~\ref{appendix:proof} shows that 
\begin{equation}
    \sigma_A \geq \frac{\lambda_{\min}^+(A^TA)}{L_{\max}},
\end{equation}
where $\lambda_{\min}^+(A^TA) \in \mathbb{R}$ is the smallest eigenvalue of $A^TA$. The condition number of the problem then appears in the $L_{\max}\left(\sigma_i^{-1} + \sigma_j^{-1}\right)$ term whereas the other terms are strictly related to the topology of the graph. Parameter $\theta$ is invariant to the scale of $\mu$ because rescaling  $\mu$ would also multiply $\lambda_{\min}^+(A^TA)$ by the same constant. The $p_{ij}^2 / (\sigma_i^{-1} + \sigma_j^{-1})$ term indicates that non-smooth edges should be sampled more often, and the square root dependency is consistent with known results for accelerated coordinate descent methods~\citep{allen2016even, nesterov2017efficiency}. If both sampling probabilities and smoothnesses are fixed, the $\mu_{ij}$ terms can be used to make the dual coordinate (which corresponds to the edge) smoother so that larger step sizes can be used to compensate for the fact that they are only rarely updated. Yet, this may decrease the spectral gap of the graph and slow convergence down.

\begin{proof}
The proof consists in evaluating $\|v_{t+1} - x^*\|^2_{A^+A}$ and follows the same scheme as by \citet{nesterov2017efficiency}. However, $F_A^*$ is not strongly convex because matrix $A^TA$ is generally not full rank. Yet, $F_A^*$ is strongly convex for the pseudo-norm $A^+A$ and the value of $F_A^*(x)$ only depends on the value of $x$ on ${\rm Ker}(A)^\perp$. \citet{gower2018accelerated} develop a similar proof in the quadratic case but without assuming any specific structure on $A$. The detailed proof can be found in Appendix~\ref{appendix:proof}.
\end{proof}

\subsection{Practical algorithm}
Algorithm~\ref{algo:adacd} is written in a form that is convenient for analysis but it is not practical at all. Its logically equivalent implementation is described in Algorithm~\ref{algo:adacd_eff}. All nodes run the same procedure with a different rank $r$ and their own local functions $f_r$ and variables $\theta_r$, $v_t(r)$ and $y_t(r)$. For convenience, we define $B = \begin{pmatrix} 1 - \theta & \theta \\ \delta & 1 - \delta \end{pmatrix}$ and $s_{ij} = \begin{pmatrix}
\frac{\theta \mu_{ij}^2}{p_{ij} \sigma_A} & \mu_{ij}^2 \eta_{ij}
\end{pmatrix}^T$.

\begin{algorithm}
\caption{Asynchronous Decentralized Accelerated Coordinate Descent}
\label{algo:adacd_eff}
\begin{algorithmic}[1]
\STATE $r$ \COMMENT{Id of the node}
\STATE $seed$ \COMMENT{The common seed}

\STATE $z_r = 0$, $y_0(r) = 0$, $v_0(r) = 0$, $t = 0$
\STATE Initialize random generator with $seed$
\WHILE{$t < T$}
\STATE Sample $e$ from $P$
\IF{$\exists j \ / \ e \in \{(r,j), (j,r)\}$}
\STATE $\begin{pmatrix} v_t(r)^T \\ y_t(r)^T \end{pmatrix}_r = B^{t - t_r} \begin{pmatrix} v_{t_r}(r)^T \\ y_{t_r}(r)^T \end{pmatrix}$
\STATE $z_r = \nabla f^*_r\left(y_t(r)\right)$
\STATE $send\_gradient(x_r, j)$ \hspace{15pt} \COMMENT{\emph{non blocking}}
\STATE $z_{dist} = receive\_gradient(j)$ \ \ \COMMENT{\emph{blocking}}
\STATE $g_t(r) = s_e \left(z_r - z_{dist}\right)$ 
\STATE $\begin{pmatrix} v_{t + 1}(r)^T \\ y_{t + 1}(r)^T \end{pmatrix}_r = B \begin{pmatrix} v_{t}(r)^T \\ y_{t}(r)^T \end{pmatrix} - g_t(r)^T$
\STATE $t_r = t + 1$
\ENDIF
\STATE t = t + 1
\ENDWHILE
\STATE \textbf{return} $z_r$
\end{algorithmic}
\end{algorithm}

Note that each update only involves two nodes, thus allowing for many updates to be run in parallel. Algorithm~\ref{algo:adacd_eff} is obtained by multiplying the updates of Algorithm~\ref{algo:adacd} by $A$ on the left. This has the benefit of switching from edge variables (of size $E \times d$) to node variables (of size $n \times d$). Then, if $y_t$ corresponds to the variable of Algorithm~\ref{algo:adacd}, $y_t(i) = e_i^TAy_t$ represents the local $y_t$ variable of node $i$ and is used to compute the gradient of $f^*_i$. We obtain $v_t(i)$ in the same way. The updates can be expressed as a matrix multiplication (contraction step, making $y_t$ and $v_t$ closer), plus a gradient term which is equal to $0$ if the node is not at one end of the sampled edge. The multiplication by $B^{t - t_r}$ corresponds to catching up the global contraction steps for updates in which node $r$ did not take part. The form of $s_{ij}$ comes from the fact that $A e_{ij} e_{ij}^T A^T = \mu_{ij}^2 (e_i - e_j)(e_i - e_j)^T$.

\subsection{Communication schedule}
\label{sec:com_schedule}
 Even though updates are actually local, nodes need to keep track of the total number of updates performed (variable $t$) in order to properly execute Algorithm~\ref{algo:adacd_eff}.

This problem can be handled by generating in advance the sequence of all communications and then simply unrolling this sequence as the algorithm progresses. All nodes perform the neighbors selection protocol starting with the same seed and only consider the communications they are involved in. Therefore, they can count the number of iterations completed.

This way of selecting neighbours can cause some nodes to wait for the gradient of a busy node before they can actually perform their update. Since the communication schedule is defined in advance, they cannot choose a free neighbor and exchange with him instead. However, any way of making edges sampled independent from the previous ones would be equivalent to generating the sequence in advance. Indeed, choosing free neighbors over busy ones would introduce correlations with the current state and therefore with the edges sampled in the past.

\section{Performances}
\subsection{Homogeneous decentralized networks}
In this section, we introduce two network-related assumptions under which the performances of \esdacd~are provably comparable to the performances of randomized gossip averaging or \ssda. We denote $p_{\max} = \max_{ij} p_{ij}$ and $p_{\min} = \min_{ij} p_{ij}$. We also note $\bar{p}(\mathcal{G}) = \max_i p_i$ and $\underbar{p}(\mathcal{G}) = \min_i p_i$ the maximum and minimum probabilities of nodes of a graph $\mathcal{G}$ where $p_i = \sum_{i=1}^n p_{ij}$. We note $d_{\max}$ and $d_{\min}$ the maximum and minimum degrees in the graph. The dependence on $\mathcal{G}$ will be omitted when clear from the context.

\begin{assumption}
\label{assumption:regular_degree_graphs}
We say that a family of graph $\mathcal{G}$ with edge weights $p$ is quasi-regular if there exists a constant $c$ such that for $n \in \mathbb{N}$, $p_{\max} \leq c p_{\min}$ and $d_{\max} \leq c d_{\min}$.
\end{assumption}

Assumption~\ref{assumption:regular_degree_graphs} is satisfied for many standard graphs and probability distribution over edges. In particular, it is satisfied by the uniform distribution for regular degree graphs. 

\begin{assumption}
\label{assumption:almost-symmetry}
The family of graphs $\mathcal{G}$ is such that there exists a constant $c$ such that for $n \in \mathbb{N}$, $\max_{ij}e_{ij}^TA^+Ae_{ij} \leq c \frac{n}{E}$ where $A$ is of the form of Equation~\eqref{eq:A_matrix} with $\mu_{ij} = 1$ and uniquely defines $\mathcal{G}(n)$. 
\end{assumption}

This second assumption essentially means that removing one edge or another should have a similar impact on the connectivity of the graph. It is verified with $c = 1$ if the graph is completely symmetric (ring or complete graph). Since $A^+A$ is a projector, $e_{ij}^T A^+A e_{ij} \leq 1$ so Assumption~\ref{assumption:almost-symmetry} holds true any time the ratio $\frac{n}{E}$ is bounded below. In particular, the grid, the hypercube, or any random graph with bounded degree respect Assumption~\ref{assumption:almost-symmetry}. 

\subsection{Average time per iteration}
\esdacd~updates are much cheaper than the updates of any global synchronous algorithm such as \ssda. However, the partial synchrony discussed in Section~\ref{sec:com_schedule} may drastically slow the algorithm down, making it inefficient to use cheaper iterations. Theorem~\ref{thm:time_per_iteration} shows that this does not happen for regular graphs with homogeneous probabilities. We note $\tau_{\max}$ the maximum delay of all edges.

\begin{theorem}
\label{thm:time_per_iteration}
If we denote $T_{\max}(k)$ the time taken by \esdacd~to perform $k$ iterations when edges are sampled according to the distribution $p$:
\begin{equation}
    \bar{\tau} = \mathbb{E}\left[\frac{1}{k} T_{\max}(k)\right] \leq c \bar{p} \tau_{\max}
\end{equation}
with a constant $c < 14$. 
\end{theorem}

The proof of Theorem~\ref{thm:time_per_iteration} is in Appendix~\ref{appendix:average_time}. Note that the constant can be improved in some settings, for example if all nodes have the same degrees and all edges have the same weight then a tighter bound $c < 4$ holds.

\begin{corollary}
\label{corr:linear_speedup}
If $\mathcal{G}$ satisfies Assumption~\ref{assumption:regular_degree_graphs} then there exists $c > 0$ such that for any $n\in \mathbb{N}$, the expected average time per iteration taken by \esdacd~in $\mathcal{G}(n)$ when edges are sampled uniformly verifies:
\begin{equation}
\label{eq:linear_speedup}
    \mathbb{E}\left[T_{\max}(k)\right] \leq c \frac{\tau_{\max}}{n} k + o(k).
\end{equation}
\end{corollary}

Corollary~\ref{corr:linear_speedup} shows that when all nodes have comparable activation frequencies then the expected time required to complete one \esdacd~iteration scales as the inverse of the number of nodes in the network. This result essentially means that the synchronization cost of locking edges does not grow with the size of the network and so iterations will not be longer on a bigger network. At any given time, a constant fraction of the nodes is actively performing an update (rather than waiting for a message) and this fraction does not shrink as the network grows. The time per iteration can be as high as $\tau_{\max}$ for some graph topologies that break Assumption~\ref{assumption:regular_degree_graphs}, \emph{e.g.}, star networks. These topologies are more suited to centralized algorithms because some nodes take part in almost all updates. 

\subsection{Distributed average consensus}
Algorithm~\ref{algo:adacd_eff} solves the problem of distributed gossip averaging if we set $f_i(\theta) = \frac{1}{2} \|\theta - c_i\|^2$. In this setting, $f_i^*(x) = \frac{1}{2}\|x + c_i\|^2 - \frac{1}{2}\|c_i\|^2$ and so $\nabla f_i^*(x) = x + c_i$. Local smoothness and strong convexity parameters are all equal to $1$.

At each round, an edge is chosen and nodes exchange their current estimate of the mean (which is equal to $e_i^T y_t + c_i$ for node $i$). Yet, they do not update it directly but they keep two sequences $y_t$ and $v_t$ that are updated according to a linear system. One step simply consists in doing a convex combination of these values at the previous step, plus a mixing of the current value with the value of the chosen neighbor.

The standard randomized gossip iteration consists in choosing an edge $(i,j)$ and replacing the current values of nodes $i$ and $j$ by their average. If we denote $\mathcal{E}_2(t)$
the second moment of the error at time $t$:
\begin{equation}
    \mathcal{E}_2(t) \leq (1 - \theta_{\rm gossip})^{2t} \mathcal{E}_2(0),
\end{equation}  
where $\theta_{\rm gossip} = \lambda_{\min}^+(\bar{W})$, with $\bar{W} = \frac{1}{E} L$ if $L$ is the Laplacian matrix of the graph \citep{boyd2006randomized}. 

\begin{corollary}
\label{crl:comparison_gossip}
If $\mathcal{G}$ satisfies Assumption~\ref{assumption:almost-symmetry} then there exists $c > 0$ such that for any $n \in \mathbb{N}$, if $\theta_{ESDACD}$ is the rate \esdacd~in $\mathcal{G}(n)$ and $\theta_{\rm gossip}$ the rate of randomized gossip averaging when edges are sampled uniformly then they verify:
\begin{equation}
    \theta_{ESDACD} \geq \frac{c}{\sqrt{n}} \sqrt{\theta_{\rm gossip}}.
\end{equation}
\end{corollary}
\vspace{-5pt}We can use tools from \citet{mohar1997some} to estimate the eigenvalues of usual graphs. In the case of the complete graph, $\theta_{\rm gossip} \approx n^{-1}$ and so $\theta_{ESDACD} \approx  \theta_{\rm gossip}$. Actually, we can show that in this case, \esdacd~iterations are exactly the same as randomized gossip iterations. In the case of the ring graph, $\theta_{\rm gossip} \approx n^{-3}$ and so $\theta_{ESDACD} \approx n^{-2}$ which is significantly better for $n$ large. For the grid graph, a similar analysis yields $\theta_{ESDACD} = O(n^{-3/2})$. Achieving this message complexity on a grid is an active research area and is often achieved with complex algorithms like geographic gossip~\citep{dimakis2006geographic}, relying on overlay networks, or \emph{LADA}~\citep{li2007location}, using lifted Markov chains~\citep{diaconis2000analysis}. Although synchronous gossip algorithms achieved this rate~\citep{oreshkin2010optimization}, finding an asynchronous algorithm that could match the rates of geographic gossip was still, to the best of our knowledge, an open area of research \citep{dimakis2010gossip}.

Therefore, \esdacd~shows improved rate compared with standard gossip when the eigengap of the gossip matrix is small. To our knowledge, this is the first time that better convergence rates of the second moment of the error are proven. Indeed, though they both show improved rates in expectation, the shift-register approach \citep{cao2006accelerated, liu2013analysis} has no proven rates for the second moment and the rates for the second moment of heavy ball gossip \citep{loizou2018accelerated} do not improve over standard randomized gossip averaging. Surprisingly, our results show that gossip averaging is best analyzed as a special case of a more general optimization algorithm that is not even restricted to quadratic objectives. Standard acceleration techniques shed a new light on the problem and allows for a better understanding of it. 

We acknowledge that the improved rates of convergence do not come for free. The accelerated gossip algorithm requires some global knowledge on the graph (eigenvalues of the gossip matrix and probability of activating each edge). Even though these quantities can be approximated relatively well for simple graphs with a known structure, evaluating them can be more challenging for more complex graphs (and can be even harder than or of equivalent difficulty to the problem of average consensus). Yet, we believe that \esdacd~as a gossip algorithm can still be practical in many cases, in particular when values need to be averaged over the same network multiple times or when computing resources are available at some time but not at the time of averaging. Such use cases can typically be encountered in sensor networks, in which the computation of such hyperparameters can be anticipated before deployment. In any case, the analysis shows that standard optimization tools are useful to analyze randomized gossip algorithms. 

\subsection{Comparison to SSDA}
The results described in Theorem~\ref{thm:adacd} are rather precise and allow for a fine tuning of the edges probabilities depending on the topology of the graph and of the local smoothnesses. However, the rate cannot always be expressed in a way that makes it simple to compare with \ssda.

\begin{corollary}
\label{crl:comparison_ssda}
Let $\mathcal{G}$ be a family of graph verifying Assumptions~\ref{assumption:regular_degree_graphs} and \ref{assumption:almost-symmetry}. There exists $c > 0$ such that: 
\begin{equation}
    \frac{\theta_{ESDACD}}{\bar{\tau}_{ESDACD}} \geq c \frac{1}{\tau_{\max}}\sqrt{\frac{\gamma}{\kappa}} = c \ \frac{\theta_{SSDA}}{\bar{\tau}_{SSDA}},
\end{equation}
where $\theta_{ESDACD}$ is the rate of \esdacd~when edges are sampled uniformly and $\theta_{SSDA}$ the rate of \ssda~when both algorithms use matrix $A$ as defined in Equation~\eqref{eq:A_matrix}.
\end{corollary}

The proof is in Appendix~\ref{appendix:comparison}. Actually, sampling does not need to be uniform but a ratio $\sqrt{p_{\min} / p_{\max}}$ would appear in the constant otherwise. The result of Corollary~\ref{crl:comparison_ssda} means that asynchrony comes almost for free for decentralized gradient descent in these cases. Indeed, both algorithms scale similarly in the network and optimization parameters. Note that in this case, we compare ESDACD and SSDA (and not MSDA) meaning that we implicitly assume that communication times are greater than computing times. This is because ESDACD is very efficient in terms of communication but not necessarily in terms of gradients.

Corollary~\ref{crl:comparison_ssda} states that the rates per unit of time are similar. Figure~\ref{fig:costs_comparison} compares the two algorithms in terms of network and computational resources usage. \ssda~iterations require all nodes to send messages to all their neighbors, resulting in a very high communication cost. \esdacd~avoids this cost by only performing local updates. \ssda~uses $n / 2$ times more gradients per iterations so both algorithms have a comparable cost in terms of gradients.

\begin{figure*}
\begin{tabular}{c|c|c|c | c}
    Algorithm & Improvement & Communications & Gradients computed & Speed
    \\ \hline
    SSDA & $\sqrt{\frac{\gamma}{\kappa_l}}$ & 2E & n & 1 \\ \hline
    ESDACD & $O\left(\frac{1}{n}\sqrt{\frac{\gamma}{\kappa_l}}\right)$ & 2 & 2 & $O\left(\frac{1}{n}\right)$
\end{tabular}
    \centering
    \caption{Per iteration costs of SSDA and ESDACD for quasi-regular graphs.}
\label{fig:costs_comparison}
\end{figure*}

At each \ssda~iteration, nodes need to wait for the slowest node in the system whereas many nodes can be updated in parallel with ESDACD. \esdacd~can thus be tuned not to sample slow edges too much, or on the opposite to sample quick edges but with highly non-smooth nodes at both ends more often.

Edge updates yield a strong correlation between the probabilities of sampling edges and the final rate. In heterogeneous cases (in terms of functions to optimize as well as network characteristics), the greater flexibility of ESDACD allows for a better fine-tuning of the parameters (step-size) and thus for better rates. %Yet, heterogeneity makes the analysis more complex and for now we validate this only through experiments.

\section{Experiments}
\subsection{ESDACD vs. gossip averaging}
The goal of this part is to illustrate the rate difference depending on the topology of the graph. We study graphs of $n$ nodes where $10 \%$ of the nodes have value 1 and the rest have value 0. Similar results are obtained with values drawn from Gaussian distributions.

Figures~\ref{fig:gossip_10x10_grid} and \ref{fig:gossip_ring_100} show that \esdacd~consistently beats standard and heavy ball gossip~\citep{loizou2018accelerated}. The clear rates difference for the ring graph shown in  Figure~\ref{fig:gossip_ring_100} illustrates the fact that \esdacd~scales far better for graphs with low connectivity. We chose the best performing parameters from the original paper ($\omega = 1$ and $\beta = 0.5$) for heavy ball gossip. \esdacd~is slightly slower at the beginning because we chose constant and simple learning rates. Choosing $B_0$ and $A_0$ from Appendix~\ref{appendix:proof} as in~\citet{nesterov2017efficiency} would lead to a more complex algorithm with better initial performances.

\begin{figure}
\subfloat[$10 \times 10$ grid.
\label{fig:gossip_10x10_grid}]{\includegraphics[width=0.5\linewidth]{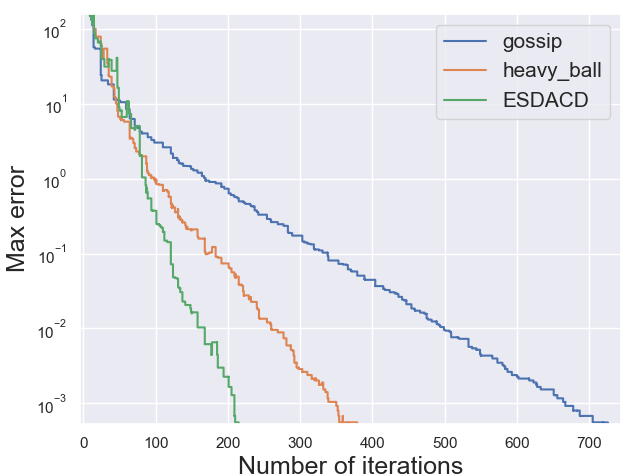}}
\subfloat[Ring graph of size $100$.
\label{fig:gossip_ring_100}]
{\includegraphics[width=0.5\linewidth]{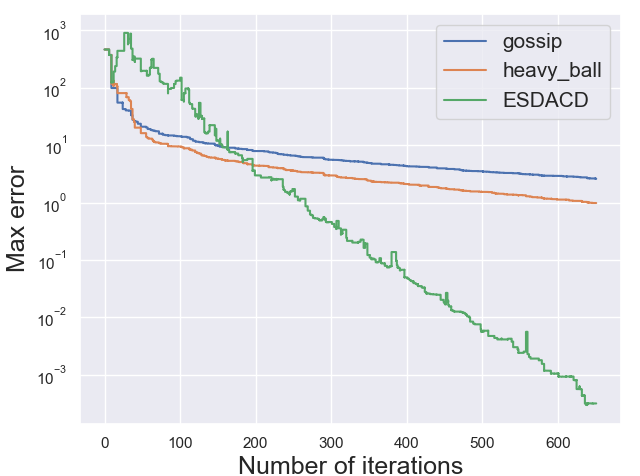}}
\caption{Comparison between ESDACD, pairwise gossip and heavy ball gossip}
\end{figure}

\subsection{ESDACD vs. SSDA}
In order to assess the performances of the algorithm in a fully controlled setting, we perform experiments on two synthetic datasets, similar to the one used by \citet{scaman2017optimal}:
\begin{itemize}
\item Regression: Each node $i$ has a vector of $N$ observations, noted $X_i \in \mathbb{R}^{d \times N}$ with $d=50$ drawn from a centered Gaussian with variance $1$. The targets $y_{i,j}$ are obtained by applying function $g: x \rightarrow \bar{x}_{i,j} + cos(\bar{x}_j) + \epsilon$ where $\bar{x}_j = d^{-1}\mathds{1}^T X_i e_j$ and $\epsilon$ is a centered Gaussian noise with variance $0.25$. At each node, the loss function is $f_i(\theta_i) = \frac{1}{2} \|X_i^T \theta - y_i\|^2 + c_i \|\theta\|^2$ with $c_i = 1$.
\item Classification: Each node $i$ has a vector of $N$ observations, noted $X_i \in \mathbb{R}^{d \times N}$ with $d=50$. Observations are drawn from a Gaussian of variance $1$ centered at $-1$ for the first class and $1$ for the second class. Classes are balanced. At each node, the loss function is $f_i(\theta_i) = \sum_{j=1}^N \ln\left(1 + \exp^{- y_{i,j} X_{i,j}^T \theta} \right) + c_i \|\theta\|^2$ with $c_i=1$.
\end{itemize}
Our main focus is on the speed of execution. Recall that edge $(i,j)$ takes time $\tau_{ij}$ to transmit a message and so if node $i$ starts its $k_i$th update at time $t_i(k_i)$ then $t_i(k_i+1) = \max_{l=i,j} t_l(k_l) + \tau_{ij}$ and the same for $j$. This gives a simple recursion to compute the time needed to execute the algorithm in an idealized setting, that we use as the x-axis for the plots.

To perform the experiments, the gossip matrix chosen for SSDA is the Laplacian matrix and $\mu_{ij}^2 = p_{ij}^{2}(\sigma_i^{-1} + \sigma_j^{-1})^{-1}$ is chosen for ESDACD. The error plotted is the maximum suboptimality $\max_i F(\theta_i) - \min_x F(x)$. Experiments are conducted on the $10 \times 10$ grid network. We perform $n / 4$ times more iteration for \esdacd~than for \ssda. Therefore, in our experiments, an execution of \ssda~uses roughly 2 times more gradients and 8 times more messages (for the grid graph) than an execution of \esdacd. This also allows us to compare the resources used by the 2 algorithms.\\

\textbf{Homogeneous setting:} In this setting, we choose uniform constant delays and $N=150$ for each node. We notice on Figure~\ref{fig:even_smoothness} that~\ssda is roughly two times faster than~\esdacd, meaning that $n / 8$ \esdacd~iterations are completed in parallel by the time \ssda~completes one iteration. This means that in average, a quarter of the nodes are actually waiting to complete the schedule, since 2 nodes engage in each iteration.\\

\begin{figure}
\subfloat[Homogeneous regression.
\label{fig:even_smoothness}]{\includegraphics[width=0.33\linewidth]{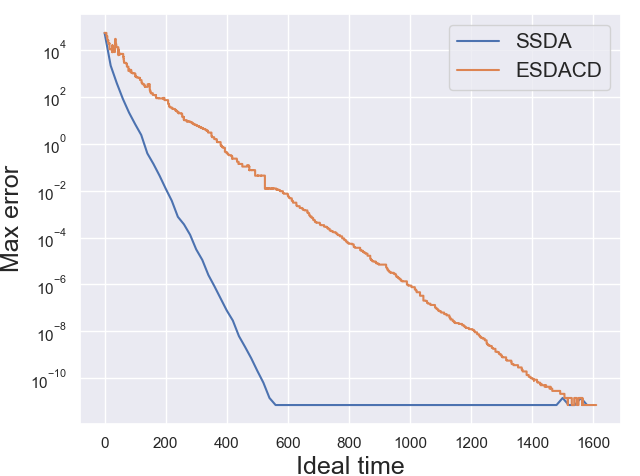}}
\subfloat[Heterogeneous regression.\label{fig:uneven_smoothness}]{\includegraphics[width=0.33\linewidth]{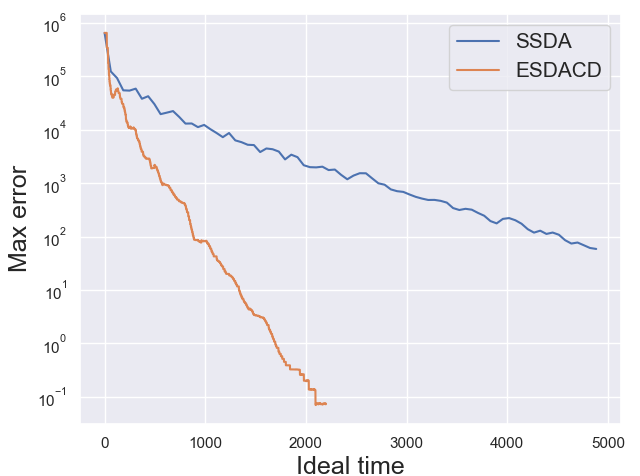}}
\subfloat[Heterogeneous classification.
\label{fig:uneven_smoothness_classif}]{\includegraphics[width=0.33\linewidth]{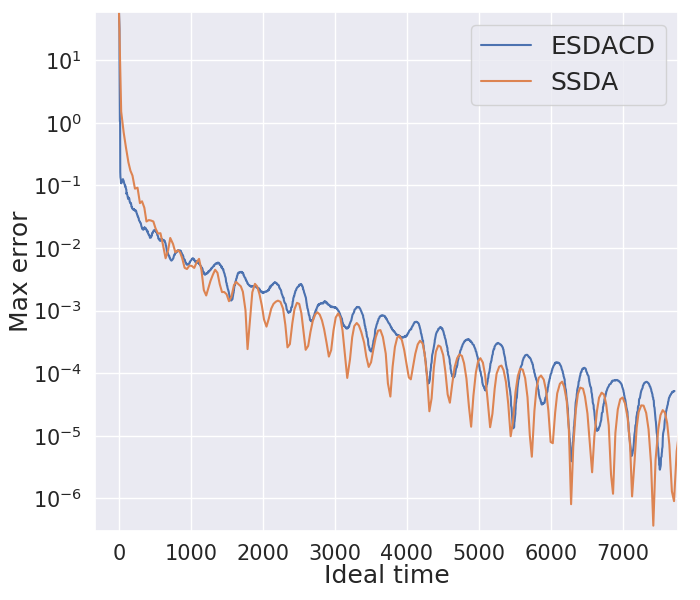}}
\caption{Comparison between the performances of ESDACD and SSDA}
\end{figure}

\textbf{Heterogeneous setting:} In this setting, $N$ is uniformly sampled between 50 (problem dimension) and 300, thus leading to very different values for the local condition numbers. Delays are all exponentially distributed with parameter $1$. Figure~\ref{fig:uneven_smoothness} shows that \esdacd~is computationally more efficient than \ssda~on the regression problem because it has a far lower final error although it uses 2 times less gradients. This can be explained by larger step sizes along regular edges and suggests that \esdacd~adapts more easily to changes in local regularity, even with uniform sampling probabilities. \esdacd~is also much faster since in average, each node performs 2 iterations in half the time needed for one \ssda~iteration. For the classification problem, strong convexity is more homogeneous because it only comes from regularization. Therefore, \esdacd~does not take full advantage of the local structure of the problem and show performances that are similar to those of \ssda.

\section{Conclusion}\vspace{-5pt}
In this paper, we introduced the \emph{Edge Synchronous Dual Accelerated Coordinate Descent} (ESDACD), a randomized gossip algorithm for the optimization of sums of smooth and strongly convex functions. We showed that it matches the performances of \ssda, its synchronous counterpart. Empirically, \esdacd~even outperforms \ssda~in heterogeneous settings. Applying \esdacd~to the distributed average consensus problem yields the first asynchronous gossip algorithm that provably achieves better rates in variance than the standard randomized gossip algorithm, for example matching the rate of geographic gossip~\citep{dimakis2006geographic} on a grid.

Promising lines of work include a communication accelerated version that would match the speed of \msda~\citep{scaman2017optimal} when computations are more expensive than communications, a fully asynchronous extension that could handle late gradients as well as a stochastic version of the algorithm that would only use stochastic gradients of the local functions.

\section*{Acknowledgement}
We acknowledge support from the European Research Council (grant SEQUOIA 724063). 
\bibliographystyle{plainnat}
\bibliography{sample}

\newpage

\appendix 
\onecolumn

\section{Detailed average time per iteration proof}
\label{appendix:average_time}
The goal of this section is to prove Theorem~\ref{thm:time_per_iteration}. The proof develops an argument similar to the one of Theorem~8.33 \citep{baccelli1992synchronization}. Yet, the theorem cannot be used directly and we need to specialize the argument for our problem in order to get a tighter bound. We note $t$ the number of iterations that the algorithm performs, and we introduce the random variable $X^t(i, w)$ such that if edge $(i,j)$ is activated at time $t+1$  (with probability $p_{ij}$), then for all $w \in \mathbb{N}^*$:
\begin{equation*}
    X^{t+1}(i,w) = X^t(i, w - 1) + X^t(j, w-1).
\end{equation*}

and $X^{t+1}(k, w - 1) = X^t(k, w - 1)$ otherwise. We start with the initial conditions $X^0(i, 0) = 1$ and $X^0(i,w) = 0$ for any $w > 0$. The following lemma establishes a relationship between the time taken by the algorithm to complete $t$ iterations and variables $X^t$.

\begin{lemma}
If we note $T_{\max}(t)$ the time at which the last node of the system finishes iteration $t$ then for all $\theta > 0$:

\begin{equation*}
    \mathbb{E}\left[T_{\max}(t)\right] \leq \theta t + \sum_{w \geq \theta t} \sum_{i=1}^n \mathds{E}\left[X^t(i, w)\right].
\end{equation*}
\end{lemma}

\begin{proof}
We first prove by induction on $t$ that if we denote $T_i(t)$ the time at which node $i$ finishes iteration $t$, then for any $i \in \{1, .., n\}$: 

\begin{equation}
\label{eq:rec_ti}
    T_i(t) = \max_{w \in \mathbb{N}, X^t(i,w) > 0} w.
\end{equation}

To ease notations, we write $w_{\max}(i,t) = \max_{w \in \mathbb{N}, X^t(i,w) > 0} w$. The property is true for $t=0$ because $T_i(0) =0$ for all $i$.

We now assume that it is true for some fixed $t > 0$ and we assume that edge $(k,l)$ has been activated at time $t$. For all $i \notin \{k, l\}$, $T_i(t+1) = T_i(t)$ and for all $w \in \mathbb{N}^*$, $X^{t+1}(i, w - 1) = X^t(i, w - 1)$ so the property is true. Besides,

\begin{align*}
    w_{\max}(k, t+1) &= \max_{w \in \mathbb{N^*}, X^t(k,w - 1) + X^t(l,w - 1) > 0} w \\
    &= \max_{w \in \mathbb{N}, X^t(i,w) + X^t(i,w) > 0} w + 1 \\
    &= 1 + \max\left(w_{\max}(k, t), w_{\max}(l, t)\right)\\
    &= 1 + \max \left(T_k(t), T_l( t)\right) = T_k(t+1).
\end{align*}

We finish the proof of Equation~\eqref{eq:rec_ti} by observing that $k$ and $l$ are completely equivalent.

The form of the recurrence guarantees that for any fixed $t \in \mathbb{N}$ and $w > 1$, if there exists $i$ such that $X^t(i, w) > 0$ then for any $w^\prime < w$, there exists $j$ such that $X^t(j, w^\prime) > 0$. Therefore, 

\begin{equation}
    T_{\max}(t) = \max_{i} \max_{w \in \mathbb{N}, X^t(i,w) > 0} w = \max_{w \in \mathbb{N}, \sum_i X^t(i,w) > 0} w =
    \sum_{w \in \mathbb{N}} \mathds{1}\left(\sum_{i=1}^n X^t(i, w) \geq 1\right),
\end{equation}

because having $X^t(i, w) > 0$ is equivalent to having $X^t(i, w) \geq 1$ since $X^t(i, w)$ is integer valued. Therefore, for any $\theta \in [0, 1]$

\begin{equation*}
    T_{\max}(t) \leq \theta t + \sum_{w \geq \theta t} \mathds{1}\left(\sum_{i=1}^n X^t(i, w) \geq 1\right),
\end{equation*}

and the proof results from taking the expectation of the previous inequality and using Markov inequality on the second term.
\end{proof}

Although there is still a maximum in the expression of $T_i(t)$, the recursion for variable $X$ has a much simpler form. In particular, we will crucially exploit its linearity. We write  $p_i = \sum_{j} p_{ij}$ and introduce $\underbar{p} = \min_i p_i$ and $\bar{p} = \max_i p_i$. We now prove the following Lemma:

\begin{lemma}
\label{lemma:sum_binom}
For all $i$, if $\delta_1 = \underbar{p}$, $\delta_2 = \bar{p}$ and $\delta = \frac{2 \delta_2 - \delta_1}{1 - 2\delta_2}$ then for all $\theta > 0$

\begin{equation}
    \sum_{p \geq \theta t} \mathbb{E}\left[X^t(i, p)\right] \leq \left(1 + \delta\right)^t \mathbb{P}\left[Binom(2\delta_2, t) \geq \theta t\right].
\end{equation}

\end{lemma}

\begin{proof}

Taking the expectation over the edges that can be activated gives:

\begin{equation}
    \mathbb{E}\left[X^{t+1}(i,w)\right] = \left(1 - p_i\right) \mathbb{E}\left[X^{t}(i,w)\right] + \sum_{j} p_{ij} \mathbb{E}\left[X^{t}(j, w - 1)\right] + p_i \mathbb{E}\left[X^{t}(i, w - 1)\right].
\end{equation}

In particular, for all $i$, $\mathbb{E}\left[X^{t+1}(i,w)\right] \leq \bar{X}^t(w)$ where $\bar{X}^0(w) = 1$ if $w = 0$ and $0$ otherwise, and:

\begin{equation}
    \bar{X}^{t+1}(w) = \left(1 - \underbar{p} \right) \bar{X}^{t}(w) + 2 \bar{p} \bar{X}^{t}(w - 1).
\end{equation}

We now introduce $\phi^t(z) = \sum_{w \in \mathbb{N}} z^w \xbar^t(w)$. A direct recursion leads to:

\begin{equation}
    \phi^t(z) = \left(1 - \underbar{p} + 2 \bar{p} z\right)^t.
\end{equation}

Then, using the fact that $\delta > 0$:

\begin{equation}
    \phi_t(z) \leq (1 + \delta)^t \left(1 - 2\delta_2 + \frac{2 \delta_2}{1 + \delta} z\right)^t \leq (1 + \delta)^t \left(1 - 2\delta_2 + 2 \delta_2 z\right)^t = (1 + \delta)^t \phi_{bin}(2\delta_2, t)(z),
\end{equation}

where $\phi_{bin}(2\delta_2, t)$ is the generating function of the Binomial law of parameters $2\delta_2$ and $t$. The inequalities above on the integral series $\phi_t$ and $(1 + \delta)^t \phi_{bin}(2\delta_2, t)$ actually hold coefficient by coefficient. Therefore, $\mathbb{E}\left[X^t(i, p)\right] \leq (1 + \delta)^t \mathbb{P}\left(Binom(2\delta_2, t) = p\right)$
\end{proof}

We conclude the proof of the theorem with this last lemma:

\begin{lemma}
If $\theta = 6\delta_2 + \delta$ then:

\begin{equation}
    \lim_{t \in \mathbb{N}} \sum_{w \geq \theta t} \mathbb{E}\left[X^t(i,w)\right] = 0
\end{equation}
\end{lemma}

\begin{proof}
We use tail bounds for the Binomial distribution \citep{arratia1989tutorial} in order to get for $\theta \geq 2 \delta_2$:

\begin{equation}
    \ln \mathbb{P}\left[Binom(2\delta_2, t) \geq \theta t\right] \leq - t D(\theta || 2 \delta_2),
\end{equation}
where $D(p ||q) = p \ln \frac{p}{q} + (1 - p)\ln\frac{1-p}{1-q}$ so applying Lemma~\ref{lemma:sum_binom} yields:

\begin{equation}
    \sum_{w \geq \theta t} \mathbb{E}\left[X^t(i,w)\right] \leq e^{-t \left[D(\theta || 2 \delta_2) - \ln(1 + \delta)\right]}.
\end{equation}

Therefore, we are left to prove that $D(\theta || 2 \delta_2) - \ln(1 + \delta) > 0$. However, 

\begin{equation}
    D(\theta || 2 \delta_2) = 2\delta_2 \ln(\frac{2\delta_2}{\theta}) + (1 - 2\delta_2) \ln\frac{1 - 2\delta_2}{1 - \theta} \geq 2\delta_2 \ln(\frac{2\delta_2}{\theta}) - 2\delta_2 + \theta 
\end{equation}

by using that $\frac{x}{1+x} \leq \ln(1 + x) \leq x$. Since $\theta = 6\delta_2 + \delta$ and $\delta \leq \frac{2 \delta_2}{1 - 2\delta_2}$, assuming that $\delta_2 \leq \frac{3}{8}$ yields:

\begin{equation}
D(\theta || 2 \delta_2) \geq 2\delta_2\left[2 - \ln(3 + \frac{\delta}{2 \delta_2})\right] + \delta > \delta \geq \ln(1 + \delta).
\end{equation}

If $\delta_2 \geq \frac{3}{8}$, then $\theta > 1$ so the result is obvious because $X^t(i,w) = 0$ for $w > t$.

\end{proof} 

\section{Execution speed comparisons}
\label{appendix:comparison}
\subsection{Comparison with gossip}
In this section, we prove Corollary~\ref{crl:comparison_gossip}.
\begin{proof}
We consider a matrix $A$ such that $A e_{ij} = \mu_{ij} (e_i - e_j)$ and $\mu_{ij}^2 = \frac{1}{2}$ for all $(i,j) \in E$. Then multiplying by $W_{ij} = Ae_{ij}e_{ij}^TA^T$ corresponds to averaging the values of nodes $i$ and $j$ and so the rate of uniform randomized gossip averaging depends on $\bar{W} = \mathbb{E}[W_{ij}]$.

In this case, applying \esdacd~with matrix $A$ yields a rate of

\begin{equation}
    \theta_{ESDACD} = \min_{ij} \frac{p_{ij}}{\mu_{ij}\sqrt{\sigma_i^{-1} + \sigma_j^{-1}}}\frac{\sqrt{\lambda^+_{\min}(A^TA)}}{\sqrt{e_{ij} A^+A e_{ij}}} \geq \sqrt{\frac{\lambda_{\min}^+(AA^T)}{c n E}}
\end{equation} 

where $c$ is a constant independent of the size of the graph coming from Assumption~\ref{assumption:almost-symmetry}.

Since $\bar{W} = \frac{1}{E} AA^T$ then $\theta_{\rm gossip} = \frac{1}{E} \lambda_{\min}^+(AA^T)$ and so:

\begin{equation}
    \theta_{ESDACD} \geq \frac{c^\prime}{\sqrt{n}} \sqrt{\theta_{\rm gossip}}
\end{equation} 

with $c^\prime=c^{-\frac{1}{2}}$.
\end{proof}

\subsection{Comparison with SSDA}
In this section, we prove Corollary~\ref{crl:comparison_ssda}. \ssda~is based on an arbitrary gossip matrix whereas the rate of \esdacd~is based on a specific matrix $A^TA$ where $Ae_{ij} = \mu_{ij}(e_i - e_j)$. Yet, $W = AA^T$ is a perfectly valid gossip matrix. Indeed, ${\rm Ker}(W) = {\rm Ker}(A) = Vec\left(\mathds{1}\right)$ and $AA^T$ is an $n\times n$ symmetric positive matrix defined on the graph $\mathcal{G}(n)$. Besides, $\lambda_{\min}^+(A^TA) = \lambda_{\min}^+(AA^T)$, which enables us to compare the rates of \ssda~and \esdacd.

\begin{proof}
For arbitrary $\mu$, the rate of \esdacd~ writes:

\begin{equation}
   \theta_{ESDACD} \geq \min_{ij} \frac{p_{ij}}{\mu_{ij} \sqrt{L_{\max}(\sigma_i^{-1} + \sigma_j^{-1}) e_{ij}^TA^+A e_{ij}}} \sqrt{\lambda_{\min}^+(A^TA)}.
\end{equation}

Here, we choose $\mu_{ij}^2= \frac{1}{2}$, which yields the bound:

\begin{equation}
\theta_{ESDACD} \geq p_{\min} \sqrt{\frac{\lambda_{\max}(AA^T)}{\max_{ij} e_{ij}^T A^+A e_{ij}}} \sqrt{\frac{\gamma}{\kappa}}.
\end{equation}

Therefore, combining this with Theorem~\ref{thm:time_per_iteration} and Assumption~\ref{assumption:almost-symmetry} gives:

\begin{equation}
    \frac{\theta_{ESDACD}}{\bar{\tau}_{ESDACD}} \geq \frac{p_{\min}\sqrt{E}}{c \bar{p} \tau_{\max}}  \sqrt{\frac{\lambda_{\max}(AA^T)}{n}}\sqrt{\frac{\gamma}{\kappa}} \geq \frac{c^\prime}{\tau_{\max}} \frac{p_{\min}}{p_{\max}}\sqrt{\frac{d_{\min}}{d_{\max}}}\sqrt{\frac{E}{n d_{\max}}} \sqrt{\frac{\gamma}{\kappa}}
\end{equation}

where we have used that $\lambda_{\max} \geq \frac{1}{n} Tr(AA^T) \geq d_{\min}$ and $\bar{p} \leq p_{\max} d_{\max}$. We then use Assumption~\ref{assumption:regular_degree_graphs} to get that there exists $c^{\prime \prime}$ such that:
\begin{equation}
    \frac{\theta_{ESDACD}}{\bar{\tau}_{ESDACD}} \geq \frac{c^{\prime \prime}}{\tau_{\max}} \sqrt{\frac{\gamma}{\kappa}} = c^{\prime \prime} \frac{\theta_{SSDA}}{\bar{\tau}_{SSDA}}
\end{equation}
\end{proof}

In the proof above, it appears that having probabilities that are too unbalanced harms the convergence rate of \esdacd. However, if these probabilities are carefully selected to match the square root of the smoothness along the edge, and if delays are such that this does not cause very slow edges to be sampled too often then unbalanced probabilities can greatly boost the convergence rate. 

\section{Detailed rate proof}
\label{appendix:proof}
The proof of Theorem~\ref{thm:adacd} is detailed in this section. Recall that we note $A^+$ the pseudo-inverse of $A$ and we define the scalar product $\langle x, y \rangle_{A^+ A} = x^T A^+ A y$. The associated norm is a semi-norm because $A^+ A$ is positive semi-definite. Since $A^+A$ is a projector on the orthogonal of ${\rm Ker}(A)$, it is a norm on the orthogonal of ${\rm Ker}(A)$. 

Our proof follows the key steps of \citet{nesterov2017efficiency}. However, we study the problem in the norm defined by $A^+A$ because our problem is strongly convex only on the orthogonal of ${\rm Ker}(A)$. Matrix $A$ can be tuned so that $F_A^*$ has the same smoothness in all directions, thus leading to optimal rates. We start by two small lemmas to introduce the strong convexity and smoothness inequalities for the $A^+A$ semi-norm. We note $U_{ij} = e_{ij} e_{ij}^T$.

\begin{lemma}[Strong convexity of $F_A^*$]
\label{lemma:sc_FA}
For all $x,y \in \mathbb{R}^E$, 
\begin{equation}
\label{eq:strong_convexity}
F_A^*(x) - F_A^*(y) \geq \nabla F_A^*(y)^T (x - y) + \frac{\sigma_A}{2} \|x - y\|^2_{A^+A}
\end{equation}
with $\sigma_A = \frac{\lambda_{\min}^+(A^TA)}{L_{\max}}$
\end{lemma}

\begin{proof}
Inequality~\eqref{eq:strong_convexity} is obtained by writing the strong convexity inequality for each $f_i^*$ and then summing them. Then, we remark that $L_i \leq L_{\max}$ for all $i$ and that $\|Aw\|^2 = \|Aw\|^2_{A^+A} \geq \lambda_{\min}^+(A^TA) \|w\|^2_{A^+A}$ for $w = x - y$. More specifically:

\begin{align*}
F_A^*(x) - F_A^*(y) &= \sum_{i=1}^n \left(f_i^*\left(e_i^TAx \right) - f_i^*\left(e_i^TAy\right)\right)\\
&\geq \sum_{i=1}^n \nabla f_i^* \left(e_i^TAy\right)^T e_i^T(Ax - Ay) + \frac{1}{2}(Ax - Ay)^T \left( \sum_{i=1}^n L_i^{-1} e_i e_i^T \right) (Ax - Ay)\\
&\geq \nabla F_A^*(y)^T (x - y) +\frac{1}{2L_{\max}} (x - y)^T A^TA (x - y)\\
&\geq \nabla F_A^*(y)^T (x - y) +\frac{\lambda_{\min}(A^TA)}{2L_{\max}} \|x - y\|^2_{A^+A}
\end{align*}
\end{proof}

\begin{lemma}[Smoothness of $F_A^*$]
\label{lemma:smoothness_FA}
We note $x_{t+1} = y_t - h_{kl} U_{kl} \nabla F_A^*(y_t)$ where $h_{kl}^{-1} = \mu_{kl}^2 (\sigma_k^{-1} + \sigma_l^{-1})$. If edge $(k,l)$ is sampled at time $t$,

\begin{equation}
\label{eq:smoothness}
F_A^*(x_{t+1}) - F_A^*(y_t) \leq - \frac{1}{2 \mu_{kl}^2 \left(\sigma_k^{-1} + \sigma_l^{-1}\right)} \| U_{kl} \nabla F_A^*(y_t) \|^2.
\end{equation}

\end{lemma}

Equation~\eqref{eq:smoothness} can be seen as an $ESO$ inequality~\citep{richtarik2016parallel} applied to the directional update $h_{kl} U_{kl} \nabla F_A^*(y_t)$.

\begin{proof}
Assuming that edge $(k,l)$ is drawn at time $t$, we use that each $f_i^*$ is ($\sigma_i^{-1}$)-smooth to write:
\begin{equation*}
f_i^*\left(e_i^TAx_{t+1}\right) - f_i^*\left(e_i^TAy_{t}\right) \leq - h_{kl} \nabla f_i^* \left(e_i^TAy_t\right)^T e_i^TA U_{kl} \nabla F_A^*(y_t) + \frac{1}{2\sigma_i} \|h_{kl} e_i^TA U_{kl} \nabla F_A^*(y_t)\|^2.
\end{equation*}

Summing it over all values of $i$ gives:

\begin{equation*}
F_A^*(x_{t+1}) - F_A^*(y_t) \leq \nabla F_A^*(y_t)^T \left[ - h_{kl} U_{kl} + \frac{1}{2} h_{kl}^2 U_{kl} A^T \sum_{i=1}^n \sigma_i^{-1} e_i e_i^T A U_{kl} \right] \nabla F_A^*(y_t).
\end{equation*}

Then, we decompose by using that $A e_{ij} = \mu_{ij} (e_i - e_j)$ and $U_{kl} = e_{kl} e_{kl}^T$ to get that 

\begin{equation*}
F_A^*(x_{t+1}) - F_A^*(y_t) \leq \nabla F_A^*(y_t)^T U_{kl} \left[ - h_{kl} + \frac{1}{2}h_{kl}^2 \mu_{kl}^2 (\sigma_k^{-1} + \sigma_l^{-1}) \right] \nabla F_A^*(y_t).
\end{equation*}

We conclude the proof by using the fact that $h_{kl} = \frac{1}{\mu_{kl}^2 (\sigma_k^{-1} + \sigma_l^{-1})}$.
\end{proof}

We can now start the proof of Theorem~\ref{thm:adacd}. We first prove the convergence of a different algorithm which is essentially the one by~\citet{nesterov2017efficiency} and show that Algorithm~\ref{algo:adacd} is obtained for a specific choice of initial conditions.

\begin{proof}
More specifically, we choose $A_0, B_0 \in \mathbb{R}$ and recursively define the following coefficients:

\begin{align}
    & \label{eq:s2} a_{t+1}^2 S^2 = A_{t+1} B_{t+1} \\
    & B_{t+1} = B_t + \sigma_A a_{t+1} \\
    & \label{eq:At1}A_{t+1} = A_t + a_{t+1} \\
    & \label{eq:alphat}\alpha_t = \frac{a_{t+1}}{A_{t+1}}\\
    & \beta_t = \frac{\sigma_A a_{t+1}}{B_{t+1}}.
\end{align}

Then, we take arbitrary $x_0, y_0, v_0 \in \mathbb{R}^{E \times d}$ and recursively define:

\begin{equation}
    y_t = \frac{(1 - \alpha_t) x_t + \alpha_t(1 - \beta_t)v_t}{1 - \alpha_t \beta_t}
\end{equation}
\begin{equation}
    \label{eq:vtplus1}
    v_{t+1} = (1 - \beta_t)v_t + \beta_t y_t - \frac{a_{t+1}}{B_{t+1} p_{ij}} U_{ij} \nabla F_A^*(y_t)
\end{equation}
\begin{equation}
\label{eq:xtplus1}
    x_{t+1} = y_t - \frac{1}{\mu_{ij}^2(\sigma_i^{-1} + \sigma_j^{-1})}U_{ij} \nabla F_A^*(y_t).
\end{equation}

For convenience, we write $w_t = (1 - \beta_t) v_t + \beta_t y_t$. Then, we study the quantity $r_t^2 = \| v_t - x^* \|^2_{A^+A}$ where $x^*$ is the minimizer of $F_A^*$. Recall that $g_{ij}(y_t) = \frac{a_{t+1}}{B_{t+1} p_{ij}} U_{ij} \nabla F_A^*(y_t)$.

\begin{equation}
\label{eq:vt1gt}
\| v_{t+1} - x^* \|^2_{A^+A} = \| w_t - x^* \|^2_{A^+A} + \|\frac{a_{t+1}}{B_{t+1} p_{ij}} U_{ij} \nabla F_A^*(y_t)\|^2_{A^+A} - 2 \frac{a_{t+1}}{B_{t+1} p_{ij}} \nabla F_A^*(y_t)^T U_{ij} A^+A (w_t - x^*).
\end{equation}
Then, 
\begin{equation}
\mathbb{E}_{ij}[\frac{a_{t+1}}{B_{t+1} p_{ij}} \nabla F_A^*(y_t)^T U_{ij}] = \sum_{ij} p_{ij} \frac{a_{t+1}}{B_{t+1} p_{ij}} \nabla F_A^*(y_t)^T U_{ij} = \frac{a_{t+1}}{B_{t+1}} \nabla F_A^*(y_t)^T.
\end{equation}

Therefore, Equation~\eqref{eq:vt1gt} can be rewritten:

\begin{equation}
\mathbb{E}\left[r_{t+1}^2\right] \leq \mathbb{E}\left[\| w_t - x^* \|^2_{A^+A}\right] + \mathbb{E}\left[\frac{e_{ij}^TA^+Ae_{ij}a_{t+1}^2}{B_{t+1}^2 p_{ij}^2} \|U_{ij} \nabla F_A^*(y_t)\|^2\right] - 2 \frac{a_{t+1}}{B_{t+1}} \nabla F_A^*(y_t)^T (w_t - x^*).
\end{equation}

Now, the goal is to write a smoothness equation to control the middle term and make $F_A^*(x_{t+1})$ appear. This control is provided by Equation~\eqref{eq:smoothness} in Lemma~\ref{lemma:smoothness_FA}.

Therefore, if we choose $S$ such that for all $(i,j)$, $\frac{e_{ij}^T A^+A e_{ij}(\sigma_i^{-1} + \sigma_j^{-1})\mu_{ij}^2}{p_{ij}^2} \leq S^2$ then the equation becomes:

\begin{equation}
\| v_{t+1} - x^* \|^2_{A^+A} \leq \| w_t - x^* \|^2_{A^+A} + \frac{2 S^2 a_{t+1}^2}{B_{t+1}^2} \left[F_A^*(y_t) - \mathbb{E}\left[ F_A^*(x_{t+1})\right] \right] - 2 \frac{a_{t+1}}{B_{t+1}} \nabla F_A^*(y_t)^T (w_t - x^*).
\end{equation}

We use the convexity of the squared norm to get that $\| w_t - x^* \|^2_{A^+A} \leq (1 - \beta_t) r_t^2 + \beta_t \| y_t - x^* \|^2_{A^+A}$. Then, if we multiply both sides by $B_{t+1}$ we get:

\begin{equation}
\label{eq:exp_inec}
B_{t+1}r_{t+1}^2 \leq B_t r_t^2 + \beta_t B_{t+1} \| y_t - x^* \|^2_{A^+A} + \frac{2 S^2 a_{t+1}^2}{B_{t+1}} \left[F_A^*(y_t) - \mathbb{E}\left[ F_A^*(x_{t+1})\right] \right] - 2 a_{t+1} \nabla F_A^*(y_t)^T (w_t - x^*).
\end{equation}

We can now use Equation~\eqref{eq:strong_convexity} of Lemma~\ref{lemma:sc_FA} (strong convexity of $F_A^*$ in norm $A^+A$) to write that:

\begin{align*}
- a_{t+1}& \nabla F_A^*(y_t)^T (w_t - x^*) = a_{t+1} \nabla F_A^*(y_t)^T A^+A \left(x^* - y_t + \frac{1 - \alpha_t}{\alpha_t}(x_t - y_t)\right)\\
&\leq a_{t+1} \left(F_A^*(x^*) - F_A^*(y_t) - \frac{1}{2} \sigma_A \|y_t - x^*\|^2_{A^+A} + \frac{1 - \alpha_t}{\alpha_t}(F_A^*(x_t) - F_A^*(y_t)) \right) \\
&\leq a_{t+1} F_A^*(x^*) - A_{t+1} F_A^*(y_t) + A_t F_A^*(x_t) - \frac{1}{2} a_{t+1} \sigma_A \|y_t - x^*\|^2_{A^+A}.
\end{align*}

Then, we combine the previous inequality with Equation~\eqref{eq:exp_inec} and we use the fact that $B_{t+1} \beta_t = a_{t+1} \sigma_A$ so that:

\begin{equation}
B_{t+1}r_{t+1}^2 \leq B_t r_t^2 + 2 A_{t+1}\left[F_A^*(y_t) - \mathbb{E}\left[ F_A^*(x_{t+1})\right] \right] - 2\left[  (A_{t+1} - A_t) F_A^*(x^*) - A_{t+1} F_A^*(y_t) + A_t F_A^*(x_t) \right],
\end{equation}

and so:

\begin{equation}
B_{t+1}r_{t+1}^2 - B_t r_t^2 \leq 2 A_t \left[F_A^*(x_t) - F_A^*(x^*)\right] - 2 A_{t+1}\left[\mathbb{E}\left[ F_A^*(x_{t+1})\right] - F_A^*(x^*) \right].
\end{equation}

By summing over all inequalities, we get that

\begin{equation}
2 A_t \mathbb{E}\left[F_A^*(x_{t}) - F_A^*(x^*)\right] + B_t \mathbb{E}[r_t^2] \leq r_0^2.
\end{equation}

Now, we need to estimate the growth of coefficients $A_t$ and $B_t$. We prove by induction on $t$ that if $A_0 = 1$ and $B_0 = \sigma_A$ then for all $t \in \mathbb{N}$, $\alpha_t = \beta_t = \frac{\sqrt{\sigma_A}}{S}$ $A_t = \left(1 - \frac{\sqrt{\sigma_A}}{S}\right)^{-t}$ and $B_t = \sigma_A A_t$.

We can first combine Equation~\eqref{eq:At1} and Equation~\eqref{eq:alphat} to obtain 

\begin{align}
    &\label{eq:alphak} a_{t+1}(\alpha_t^{-1} - 1) = A_t\\
    &\label{eq:betak} a_{t+1}(\beta_t^{-1} - 1) = \frac{B_t}{\sigma_A}
\end{align}

For $t=0$, we can combine equations~\eqref{eq:alphak} and \eqref{eq:betak} to obtain that $\alpha_0^{-1} - 1 = \beta_0^{-1} - 1$ (since $a_1 \neq 0$ and so $\alpha_0 = \beta_0$. Finally, 

\begin{equation*}
    a_1^2 S^2 = A_1 B_1 = \frac{a_1^2 \sigma_A}{\alpha_0 \beta_0}
\end{equation*}

and so $\alpha_0 = \beta_0 = \frac{\sqrt{\sigma_A}}{S}$.

Now suppose that the property is true for a given $t \geq 0$. Then, we use Equation~\eqref{eq:alphak} and the fact that $A_{t+1} = a_{t+1} + A_t$. Since $1 + (\alpha_t^{-1} - 1)^{-1} = \frac{\alpha_t^{-1} - 1 + 1}{\alpha_t^{-1} - 1} = (1 - \alpha_t)^{-1}$ then by induction assumption, $A_{t+1} = \left(1 - \frac{\sqrt{\sigma_A}}{S}\right)^{-t - 1}$. 

We use Equation~\eqref{eq:betak} in the same way to prove that $B_{t+1} = \sigma_A A_{t+1}$. 

Then, we use equations~\eqref{eq:alphak} and \eqref{eq:betak} at time $t+1$ to get that $\alpha_{t+1}^{-1} - 1 = \beta_{t+1}^{-1} - 1$ so $\alpha_{t+1} = \beta_{t+1}$. Their value can again be retrieved by using Equation~\eqref{eq:s2}, which finishes the induction.

We have proven that for this choice of $A_0$ and $B_0$ the $\alpha$ and $\beta$ coefficients are constant and are equal to $\theta = \frac{\sqrt{\sigma_A}}{S}$. Therefore, $v_{t+1} = (1 - \theta)v_t + \theta y_t - \frac{\theta}{p_{ij}\sigma_A} U_{ij} \nabla F_A^*(y_t)$. With this choice of parameters, $y_{t+1}$ can be expressed as:

\begin{equation*}
y_{t+1} = \frac{(1 - \theta)x_{t+1} + \theta(1 - \theta)v_{t+1}}{1 - \theta^2} = \frac{x_{t+1} + \theta v_{t+1}}{1 + \theta}.
\end{equation*}

Then, the coefficients of Algorithm~\ref{algo:adacd} are recovered by replacing $x_{t+1}$ and $v_{t+1}$ by their expressions in Equations~\eqref{eq:xtplus1} and \eqref{eq:vtplus1}. The actual values of $a_{t+1}$, $A_{t+1}$ and $B_{t+1}$ are only used for the analysis because only $\frac{a_{t+1}}{B_{t+1}} = \frac{\sigma_A}{\beta_t}$ appears in the recursion. 
\end{proof}

\end{document}